\documentclass[a4paper,twoside]{amsart}
\usepackage{xypic,amscd,amssymb,latexsym,srcltx}
\usepackage{amsthm}
\def\@setcopyright{}
\makeatother

\newcommand{\Y}{\operatorname{\mathcal Y}\nolimits}

\newcommand{\C}{\gamma}

\newtheorem{lemma}{Lemma}[section]
\newtheorem{proposition}[lemma]{Proposition}

\newtheorem{theorem}[lemma]{Theorem}
\newtheorem{remark}[lemma]{Remark}
\newtheorem{conjecture}[lemma]{Conjecture}
\newtheorem*{acknowledgments*}{ACKNOWLEDGMENTS}
\newtheorem{definition}[lemma]{Definition}

\newtheorem*{conj*}{Conjecture}
\newtheorem*{thm*}{Main Theorem}
\newtheorem*{remark*}{Remark}
\newtheorem*{lemma*}{Lemma}
\newtheorem*{example*}{[Example]}

\begin{document}
\begin{center}
{\Large \bf A topological interpretation about $m_{G, N}$ for finite group $G$ with normal subgroup $N$
}\\

\bigskip
\end{center}

\begin{center}
Heguo Liu$^{1}$, Xingzhong Xu$^{1, 2}$
, Jiping Zhang$^{3}$

\end{center}

\footnotetext {$^{*}$~~Date: 05/03/2018.
}
\footnotetext {
1. Department of Mathematics, Hubei University, Wuhan, 430062, China

2. Departament de Matem$\mathrm{\grave{a}}$tiques, Universitat Aut$\mathrm{\grave{o}}$noma de Barcelona, E-08193 Bellaterra,
Spain

3. School of Mathematical Sciences, Peking University, Beijing, 100871, China

Heguo Liu's E-mail: ghliu@hubu.edu.cn

Xingzhong Xu's E-mail: xuxingzhong407@mat.uab.cat; xuxingzhong407@126.com

Jiping Zhang's E-mail: jzhang@pku.edu.cn

Supported by National 973 Project (2011CB808003) and NSFC grant (11371124, 11501183).
}

\title{}
\def\abstractname{\textbf{Abstract}}

\begin{abstract}\addcontentsline{toc}{section}{\bf{English Abstract}}
Let $G$ be a finite group and $N\unlhd G$.
In this note, we construct a class poset of $G$ for some cyclic subgroup $C$ of $G$.
And we find a relation between $m_{G,N}$ and the Euler characteristic of some nerve spaces of these posets.
(see Main Theorem).
\hfil\break


\end{abstract}

\maketitle

\section{\bf Introduction}

In \cite{Bo1}, Bouc proposed the following conjecture:
\begin{conjecture}\cite[Conjecture A]{Bo1} Let $G$ be a finite group. Then $\beta(G)$ is nilpotent if
and only if $G$ is nilpotent.
\end{conjecture}
Here, $\beta(G)$ denotes a largest quotient of a finite group which is a $B$-group and the definition of $B$-group
can be found in \cite{Bo1, Bo2} or in Section 2.
Bouc has proven Conjecture 1.1  under the additional assumption that the finite group $G$ is solvable in \cite{Bo1}.
In \cite{XZ}, Xu and Zhang consider some special cases when the finite group $G$ is not solvable. But this result relies on
the proposition of Baumann \cite{Ba}, and his proposition relies on the Conlon theorem \cite[(80.51)]{CR}.
If we want to generalize the result of \cite{XZ}, we need to use the new method to compute $m_{G, N}$ directly. Here, $N$
is a normal subgroup of $G$. And the definition of $m_{G,N}$ can be found in \cite{Bo1,Bo2} or in Section 2.

For the computation of $m_{G,N}$, Bouc had computed the following first.
\begin{proposition}\cite[Proposition 5.6.1]{Bo2} Let $G$ be a finite group. Then
$$m_{G,G}=\left\{ \begin{array}{ll}
0, &\mbox{if}~ G ~is~not~cyclic;
\\[2ex] \frac{1}{|G|}\varphi(|G|), &\mbox{if} ~  G ~is ~cyclic.\end{array}\right.$$
where $\varphi$ is the Euler totient function.
\end{proposition}


For general case, our main theorem can picture $m_{G,N}$ as follows.
\begin{thm*}Let $G$ be a finite group, $G$ not cyclic and $N\unlhd G$.
Then
$$ m_{G,N}=\frac{1}{|G|}\sum_{\substack{C\leq G\\  C~ \mathrm{is~ cyclic}}}\sum_{i=1}^n
\sum_{\substack{\sigma\leq J\\ |\sigma|=i\\  C\leq H_\sigma}}
(-1)^i\widetilde{\chi}(|N(\mathfrak{T}_{C}(G, H_\sigma))|)\cdot \varphi(|C|).$$
Here, we explain the symbols of the above formula as follows:

(1) Let $\{H_1, H_2,\cdots, H_n\}$ be the set of all maximal subgroup
of $G$ such that $N\lneq H_i$;

(2) Set $J=\{1,2,\ldots, n\}$ and $\sigma$ be a non-empty subset of $J$. And $|\sigma|$ means
the order of set $\sigma$;

(3) Set $H_{\sigma}=\bigcap_{i\in\sigma}H_i$;

(4) Let $|N(\mathfrak{T}_{C}(G,H_\sigma))|$ be a simplicial complex associated to the poset $\mathfrak{T}_{C}(G,H_\sigma)$, and
$\widetilde{\chi}(|N(\mathfrak{T}_{C}(G,H_\sigma))|)$ be the reduced Euler characteristic of the space $|N(\mathfrak{T}_{C}(G, H_\sigma))|$.
\end{thm*}
Here, $\mathfrak{T}_{C}(G,H)$ is defined as follows:
 Let $C$ be a cyclic subgroup of $H$, define
$$\mathfrak{T}_{C}(G, H):=\{X|C\leq X\lneq G, X\nleq H\}.$$
We can see that $\mathfrak{T}_{C}(G, H)$ is a poset ordered by inclusion. We can consider poset $\mathfrak{T}_{C}(G, H)$
as a category with one morphism $A\rightarrow B$ if $A$ is a subgroup of $B$. We set $N(\mathfrak{T}_{C}(G, H))$
is the nerve of the category $\mathfrak{T}_{C}(G, H)$ and $|N(\mathfrak{T}_{C}(G, H))|$ is  the geometric realization
of $N(\mathfrak{T}_{C}(G, H))$.

After recalling the basic definitions and properties of $B$-groups in Section 2, we introduce a
lemma about M$\mathrm{\ddot{o}}$bius function in Section 3. And this lemma will be used
 in Section 4 to prove  Proposition 4.2.
In Section 5, we construct a class poset
$\mathfrak{T}_{C}(G, H)$ of $G$ for some cyclic subgroup $C$ and prove Main Theorem in Section 6.

\section{\bf The Burnside rings and $B$-groups}

In this section we collect some known results about the Burnside rings and $B$-groups.  For the background theory of Burnside rings and
$B$-groups, we refer to \cite{Bo1}, \cite{Bo2}.

\begin{definition}\cite[Notation 5.2.2]{Bo2} Let G be a finite group and $N\unlhd G$. Denote by $m_{G,N}$ the rational number defined by:
$$m_{G,N}=\frac{1}{|G|}\sum_{XN=G} |X|\mu(X, G),$$
where $\mu$ is the $M\ddot{o}bius$ function of the poset of subgroups of $G$.
\end{definition}

\begin{remark} If $N=1$, we have
$$m_{G,1}=\frac{1}{|G|}\sum_{X1=G} |X|\mu(X, G)=\frac{1}{|G|}|G|\mu(G, G)=1\neq 0.$$
\end{remark}

\begin{definition}\cite[Definition 2.2]{Bo1} The finite group $G$ is called a $B$-group if
$m_{G,N}=0$ for any non-trivial normal subgroup $N$ of $G$.
\end{definition}

\begin{proposition}\cite[Proposition 5.4.10]{Bo2} Let $G$ be a finite group. If $N_{1}, N_{2}\unlhd G$
are maximal such that $m_{G,N}\neq 0$, then $G/N_{1}\cong G/N_{2}$.
\end{proposition}

\begin{definition}\cite[Notation 2.3]{Bo1} When $G$ is a finite group, and $N\unlhd G$
is maximal such that $m_{G,N}\neq 0$, set $\beta(G)=G/N.$
\end{definition}

\begin{theorem}\cite[Theorem 5.4.11]{Bo2} Let $G$ be a finite group.

1. $\beta(G)$ is a $B$-group.

2. If a $B$-group $H$ is isomorphic to a quotient of $G$, then $H$ is isomorphic to a quotient of $\beta(G)$.

3. Let $M\unlhd G$. The following conditions are equivalent:

\quad\quad (a) $m_{G,N}\neq 0$.

\quad\quad (b) The group $\beta(G)$ is isomorphic to a quotient of $G/M$.

\quad\quad (c) $\beta(G)\cong \beta(G/N)$.
\end{theorem}

We collect some properties of $m_{G,N}$ that will be needed later.

\begin{proposition}\cite[Proposition 2.5]{Bo1} Let $G$ be a finite group.
Then $G$ is a $B$-group if and only if $m_{G,N}=0$ for any minimal (non-trivial) normal subgroup of $G$.
\end{proposition}

\begin{proposition}\cite[Proposition 5.6.1]{Bo2} Let $G$ be a finite group. Then
$m_{G,G}=0$ if and only if $G$ is not cyclic.
If $G$ be cyclic of order $n$,
then $m_{G, G}=\varphi(n)/n$, where $\varphi$ is the Euler totient function.
\end{proposition}

\begin{remark} If $G$ is a finite simple group, then $G$ is a $B$-group if and only if
$G$ is not abelian.
\end{remark}

We collect two results that are the relations about $G$ and $\beta(G)$.

When $p$ is a prime number, recall that a finite group $G$ is called cyclic modulo $p$ (or $p$-hypo-elementary) if
$G/O_{p}(G)$ is cyclic. And M. Baumann has proven the following theorem.

\begin{theorem}\cite[Theorem 3]{Ba} Let $p$ be a prime number and $G$ be a finite group. Then $\beta(G)$ is cyclic modulo $p$ if
and only if $G$ is cyclic modulo $p$.
\end{theorem}

In \cite{Bo1}, S. Bouc has proven the Conjecture under the additional assumption that finite group $G$ is solvable.

\begin{theorem}\cite[Theorem 3.1]{Bo1} Let $G$ be a solvable finite group. Then $\beta(G)$ is nilpotent if
and only if $G$ is nilpotent.
\end{theorem}

\section{\bf The  M$\mathrm{\ddot{o}}$bius function of the posets of groups}
In this section, we introduce a lemma about the M$\mathrm{\ddot{o}}$bius function.
In fact, this Lemma
will be used in computing $m_{G,N}$ where $G$ is a finite group and $N\unlhd G$.

Let $G$ be a finite group and let $\mu$ denote the M$\mathrm{\ddot{o}}$bius function of
subgroup lattice of $G$. We refer to \cite[p.94]{Y}:

Let $K,D\leq G$, recall the Zeta function of $G$ as following:
$$\zeta(K, D)=\left\{ \begin{array}{ll}
1, &
\mbox{if}~ K\leq D;
\\[2ex] 0, &\mbox{if} ~K\nleq D.\end{array}\right.$$

Set $n:=|\{K|K\leq G\}|$, we have a $n\times n$ matrix $A$ as following:
$$A:=(\zeta(K, D))_{K,D\leq G}.$$
It is easy to find that $A$ is an invertible matrix, so there exists $A^{-1}$ such that
$$AA^{-1}=E,$$
Here, $E$ is an identity element. Recall the M$\mathrm{\ddot{o}}$bius function as following:
$$(\mu(K, D))_{K,D\leq G}=A^{-1}.$$

Now, we set the subgroup lattice of $G$ as following:
$$\{K|K\leq G\}:=\{1=K_1, K_2,\ldots, K_n=G\}$$
where $n=|\{K|K\leq G\}|$.

Let us list the main lemma of this section as follows, and this lemma is used to prove Proposition 4.1 in Section 4.
And the proof of this lemma is due to the referee of \cite{LXZ}, and this proof is  easier than the proof of \cite[Section 3]{LXZ}.
\begin{lemma}Let $G$ be a finite group. Let $\{K_i|i=1,2,\ldots, n\}$ be the set of all subgroups of $G$.
 And Set $K_1=1, K_n=G$. Then we have $\mu(K_i, K_{i'})=0$ if $K_i\nleq K_{i'}$.
\end{lemma}

\begin{proof}Set the poset $\{K_i|i=1,2,\ldots, n\}:=X$, one defines the
incidence (or M$\ddot{\mathrm{o}}$bius) algebra $A_X$ of $X$ at the set of square matrices m
indexed by $X\times X$, with integral coefficients, such that
$$\forall~ (x,y)\in X\times X, m(x,y) \neq 0 \Longrightarrow  x \leq  y ~in~ X.$$
This defines clearly a unital subalgebra of the algebra of all square matrices
indexed by $X\times X$, with integral coefficients. The incidence matrix $\zeta_X$ of
$X$ belongs to $A_X$. Moreover, since $\zeta_X$ is unitriangular (up to a suitable
permutation of $X$), the matrix $\zeta_X-\mathrm{Id}$ is nilpotent. Hence
$$\zeta_X^{-1}=\mathrm{Id}+(\zeta_X-\mathrm{Id})^{-1}=\sum_{i=1}^{+\infty}(-1)^i(\zeta_X-\mathrm{Id})^i,$$
and the summation is actually finite. Since $\zeta_X- \mathrm{Id}\in A_X$, it follows that
$\zeta_X^{-1}\in A_X$, so $\mu(x,y) \neq 0$ implies $x\leq y$, for any $x,y\in X.$ This proves the
lemma.
\end{proof}

\section{\bf To Compute the $M'_{G, N}$}

In \cite{LXZ},
we compute $m_{G,N}$ when $|G:N|=p$ for some prime number, and we have the following observation:
\begin{eqnarray*}
&~&m_{G,N}+\frac{1}{|G|}\sum_{X\leq N} |X|\mu(X, G)\\
&=&\frac{1}{|G|}\sum_{\substack{XN= G\\ X\leq G}} |X|\mu(X, G)+\frac{1}{|G|}\sum_{X\leq N} |X|\mu(X, G)\\
&=&\frac{1}{|G|}\sum_{\substack{XN= G\\ X\leq G}} |X|\mu(X, G)+\frac{1}{|G|}\sum_{\substack{XN\neq G\\ X\leq G}} |X|\mu(X, G)\\
&=&\frac{1}{|G|}\sum_{X\leq G} |X|\mu(X, G)\\
&=&m_{G,G}=0, \mathrm{if}~ G \mathrm{~is ~not~ cyclic}.
\end{eqnarray*}
So to compute $m_{G,N}$, we can compute $\frac{1}{|G|}\sum_{X\leq N} |X|\mu(X, G)$ first.
Now, we set
$$m_{G,N}':=\frac{1}{|G|}\sum_{\substack{XN\neq G\\ X\leq G}} |X|\mu(X, G)=\frac{1}{|G|}\sum_{X\leq N} |X|\mu(X, G);$$
and set
$$M_{G,N}':=\sum_{X\leq N} |X|\mu(X, G)=|G|m_{G,N}'.$$
In \cite{LXZ}, we gave a relation between $M_{G,N}'$ and the Euler characteristic of the nerve space
of some class poset of the group $G$. But the condition $|G:N|=p$ is so strong. Under the suggestions of the referee report
 of \cite{LXZ},
we try to get rid of this condition, and we can get the following propositions.
And the reason of this section
can be found in Remark 4.3.

Now, let $G$ be a finite group and $N \lneq G$. We set
$$m_{G,N}':=\frac{1}{|G|}\sum_{X\leq N} |X|\mu(X, G);$$
and set
$$M_{G,N}':=\sum_{X\leq N} |X|\mu(X, G)=|G|m_{G,N}'.$$
Thus we have the following propositions.
\begin{proposition}Let $G$ be a finite group and $N\lneq G$.
Then
$$
M_{G,N}'
=-\sum_{Y\lneq G}\sum_{X\leq N\cap Y} |X|\mu(X, Y)$$
\end{proposition}

\begin{proof}Since $N$ is a proper subgroup
of finite group $G$, thus
if $X \leq N$, then $X$ is a proper subgroup of $G$. Then, by standard properties
of the M$\mathrm{\ddot{o}}$bius function
$$(\ast)~~~~~~~~~~\sum_{X\leq Y\leq G}\mu(X, Y)=0=\mu(X,G)+\sum_{X\leq Y\lneq G}\mu(X, Y).$$
It follows that
$$\mu(X,G)=-\sum_{X\leq Y\lneq G}\mu(X, Y).$$
Reporting this value in the definition of $M'_{G,N}$
gives
$$M_{G,N}'=-\sum_{X\leq N} |X|\sum_{X\leq Y\lneq G}\mu(X, Y).$$
By Lemma 3.1, we can see if $X\nleq Y$, then $\mu(X,Y)=0$. So we have
\begin{eqnarray*}
\sum_{X\leq N} |X|\sum_{X\leq Y\lneq G}\mu(X, Y)
&=&\sum_{X\leq N} |X|\sum_{Y\lneq G}\mu(X, Y)\\
&=&\sum_{Y\lneq G} \sum_{X\leq N}|X|\mu(X, Y)\\
&=&\sum_{Y\lneq G}\sum_{X\leq N\cap Y} |X|\mu(X, Y).
\end{eqnarray*}
Hence, we have
$$
M_{G,N}'
=-\sum_{Y\lneq G}\sum_{X\leq N\cap Y} |X|\mu(X, Y).$$
\end{proof}

\begin{proposition}Let $G$ be a finite group and $N\lneq G$.
Then
$$M_{G,N}'=-\sum_{\substack{C\leq N \\ C~ is~ cyclic}}\varphi(|C|)-\sum_{\substack{Y\lneq G\\ Y\nleq N}}M'_{Y, Y\cap N}.$$
\end{proposition}

\begin{proof}By Proposition 4.1, we have
$$M_{G,N}'
=-\sum_{Y\lneq G}\sum_{X\leq N\cap Y} |X|\mu(X, Y)$$

We will compute $\sum_{X\leq N\cap Y} |X|\mu(X, Y)$
by considering the cases when $Y\leq N$ and $Y\neq N$ in the following.

\textbf{Case 1.} $Y\leq N$. We have
\begin{eqnarray*}
\sum_{X\leq N\cap Y} |X|\mu(X, Y)
&=&\sum_{X\leq Y} |X|\mu(X, Y)\\
&=&|Y|m_{Y,Y}.
\end{eqnarray*}
If $Y$ is not cyclic, we have $m_{Y,Y}=0$. If
$Y$ is cyclic, we have $m_{Y,Y}=\frac{\varphi(|Y|)}{|Y|}$.
Hence, we have $$\sum_{X\leq N\cap Y} |X|\mu(X, Y)=\left\{ \begin{array}{ll}
\varphi(|Y|), &
\mbox{if}~ Y~is~ cyclic;
\\[2ex] 0, &\mbox{if} ~Y~is~not~ cyclic.\end{array}\right.$$

\textbf{Case 2.} $Y\nleq N$. Then we have
\begin{eqnarray*}
\sum_{X\leq N\cap Y} |X|\mu(X, Y)
&=&M_{Y, Y\cap N}'
\end{eqnarray*}
by the definition of $M_{Y, Y\cap N}'$.

Hence,
\begin{eqnarray*}
M_{G,N}'
&=&-\sum_{\substack{Y\lneq G\\ Y\leq N\\ Y~is~cyclic}}(\sum_{X\leq N\cap Y} |X|\mu(X, Y))\\
&~&-\sum_{\substack{Y\lneq G\\ Y\leq N\\ Y~is~not~cyclic}}(\sum_{X\leq N\cap Y} |X|\mu(X, Y))\\
&~&-\sum_{\substack{Y\lneq G\\ Y\nleq N}}\sum_{X\leq N\cap Y} |X|\mu(X, Y)\\
&=&-\sum_{\substack{Y\lneq G\\ Y\leq N\\ Y~is~cyclic}} \varphi(|Y|)
-(\sum_{\substack{Y\lneq G\\ Y\leq N\\ Y~is~not~cyclic}} 0)
-\sum_{\substack{Y\lneq G\\ Y\nleq N}}M_{Y, Y\cap N}'\\
&=&-\sum_{\substack{C\leq N\\ C~ is~ cyclic}}\varphi(|C|)-\sum_{\substack{Y\lneq G\\ Y\nleq N}}M'_{Y, Y\cap N}.
\end{eqnarray*}
\end{proof}

\begin{remark} To compute $M_{G,N}'$, we need compute $M_{Y, Y\cap N}'$ for every $Y\lneq G$. Since $Y\lneq G$,
thus we can get $M_{G,N}'$ by finite steps.
\end{remark}

\section{\bf A class poset of subgroups of $G$}

To compute $M_{Y, Y\cap N}'$ for every $Y\lneq G$, we define a new class poset of subgroups of $G$ in this section.
And we find the relation between $M_{G,N}'$ and this class poset.

\begin{definition} Let $G$ be a finite group and $N\lneq G$. Let $C$ be a cyclic subgroup of $N$, define
$$\mathfrak{T}_{C}(G, N):=\{X|C\leq X\lneq G, X\nleq N\}.$$
Here, $C\notin \mathfrak{T}_{C}(G, N)$ because $C\leq N$.
We can see that $\mathfrak{T}_{C}(G,N)$ is a poset ordered by inclusion. We can consider poset $\mathfrak{T}_{C}(G, N)$
as a category with one morphism $Y\rightarrow Z$ if $Y$ is a subgroup of $Z$. We set $N(\mathfrak{T}_{C}(G, N))$
is the nerve of the category $\mathfrak{T}_{C}(G, N)$ and $|N(\mathfrak{T}_{C}(G, N))|$ is  the geometric realization
of $N(\mathfrak{T}_{C}(G, N))$. More detail of topology can be seen in \cite{DH}.

Let $\sigma$ be a $n$-simplex of the nerve $N(\mathfrak{T}_{C}(G, N))$ and $\sigma$ not degenerate, that means we have the
following:
$$\sigma: \sigma(0)\to \sigma(1)\to \cdots \to \sigma(n)$$
where $\sigma(i)\in \mathfrak{T}_{C}(G)$ and $\sigma(i)\lneq \sigma(i+1)$ for all $i$.
\end{definition}

%
%
%

Since we use the Euler characteristic of $|N(\mathfrak{T}_{C}(G, N))|$ in Proposition 5.3, thus
we recall the definition of the Euler characteristic as follows:
\begin{definition}\cite[\S 22]{M} The Euler characteristic (or Euler number) of a finite complex $K$ is defined, classically,
by the equation
$$\chi(K) =\sum_{i}(-1)^i \mathrm{rank}(C_i(K)).$$
Said differently, $\chi(K)$ is the alternating sum of the number of simplices of $K$ in
each dimension.

One can also use the reduced Euler characteristic $\widetilde{\chi}(K)$ of $K$, defined by $\widetilde{\chi}(K)=\chi(K)-1$.
\end{definition}

\begin{proposition}Let $G$ be a finite group and $N\lneq G$.
Then
\begin{eqnarray*}
M_{G,N}'
&=&-\sum_{\substack{C\leq N\\C~is~ cyclic}}\varphi(|C|)-\sum_{\substack{Y\lneq G\\ Y\nleq N}}M'_{Y, Y\cap N}\\
&=&\sum_{\substack{C\leq N\\ C~ is~ cyclic}}\widetilde{\chi}(|N(\mathfrak{T}_{C}(G, N))|\cdot \varphi(|C|).
\end{eqnarray*}
Here, $|N(\mathfrak{T}_{C}(G, N))|$ is a simplicial complex associated to the poset $\mathfrak{T}_{C}(G, N)$, and
$\widetilde{\chi}(|N(\mathfrak{T}_{C}(G, N))|)$ is the reduced Euler characteristic of the space $|N(\mathfrak{T}_{C}(G, N))|$.
\end{proposition}

\begin{proof}By Proposition 4.2
and Remark 4.3, we have
\begin{eqnarray*}
M_{Y, Y\cap N}'
&=&-\sum_{\substack{C\leq Y\cap N\\ C~ is~ cyclic}}\varphi(|C|)-\sum_{\substack{Y_1\lneq Y\\ Y_1\nleq Y\cap N}}M'_{Y_1, Y_1\cap N}\\
&=&-\sum_{\substack{C\leq Y\cap N\\ C~ is~ cyclic}}\varphi(|C|)-\sum_{\substack{Y_1\lneq Y\\ Y_1\nleq N}}M'_{Y_1, Y_1\cap N}.
\end{eqnarray*}
Now, we can repeat the operations of
Proposition 4.2 on $M'_{Y_1, Y_1\cap N}$ by Remark 4.3.
So
\begin{eqnarray*}
&~&\sum_{\substack{Y\lneq G\\ Y\nleq N}}M'_{Y, Y\cap N}\\
&=&\sum_{\substack{Y\lneq G\\ Y\nleq N}}(-\sum_{\substack{C\leq Y\cap N\\ C~is~ cyclic}}\varphi(|C|)-
\sum_{\substack{Y_1\lneq Y\\ Y_1\nleq N}}M'_{Y_1, Y_1\cap N})\\
&=&-\sum_{\substack{Y\lneq G\\ Y\nleq N}}(\sum_{\substack{C\leq Y\cap N\\ C~is~ cyclic}}\varphi(|C|))
-\sum_{\substack{Y\lneq G\\ Y\nleq N}}\sum_{\substack{Y_1\lneq Y\\ Y_1\nleq N}}M'_{Y_1, Y_1\cap N}\\
&=&-\sum_{\substack{Y\lneq G\\Y\nleq N}}(\sum_{\substack{C\leq Y\cap N\\ C~is~ cyclic}}\varphi(|C|))
-\sum_{\substack{Y\lneq G\\Y\nleq N}}\sum_{\substack{Y_1\lneq Y\\Y_1\nleq N}}(-\sum_{\substack{C\leq Y_1\cap N\\C~is~ cyclic}}\varphi(|C|)
-\sum_{\substack{Y_2\lneq Y_1\\Y_2\nleq N}}M'_{Y_2, Y_2\cap N})\\
&=&-\sum_{\substack{Y\lneq G\\Y\nleq N}}(\sum_{\substack{C\leq Y\cap N\\ C~is~ cyclic}}\varphi(|C|))
+\sum_{\substack{Y\lneq G\\Y\nleq N}}\sum_{\substack{Y_1\lneq Y\\Y_1\nleq N}}(\sum_{\substack{C\leq Y_1\cap N\\C~is~ cyclic}}\varphi(|C|))\\
&~&-\sum_{\substack{Y\lneq G\\Y\nleq N}}\sum_{\substack{Y_1\lneq Y\\Y_1\nleq N}}\sum_{\substack{Y_2\lneq Y_1\\Y_2\nleq N}}M'_{Y_2, Y_2\cap N}\\
&~&~\\
&=&\cdots\cdots\cdots\\
&~&~\\
&=&-\sum_{\substack{C\leq N\\ C~is~ cyclic}}\sum_{i}\sum_{\sigma\in N(\mathfrak{T}_{C}(G))_{i}}(-1)^{i}\cdot \varphi(|C|))\\
&=&-\sum_{\substack{C\leq N\\ C~is~ cyclic}}\chi(|N(\mathfrak{T}_{C}(G, N))|)\cdot \varphi(|C|)).
\end{eqnarray*}
Here, $\sigma$ is a $i$-simplex of nerve $N(\mathfrak{T}_{C}(G, N))$ and $\sigma$ is not degenerate.

So, $$
M_{G,N}'=\sum_{\substack{C\leq N\\ C~is~ cyclic}}\widetilde{\chi}(|N(\mathfrak{T}_{C}(G, N))|\cdot \varphi(|C|).
$$
\end{proof}

\begin{proposition}Let $G$ be a finite group and $G$ be not cyclic. Let $N\unlhd G$ such that $|G:N|=p$ for some prime number $p$.
If the space $|N(\mathfrak{T}_{C}(G, N))|$ is contractible for each cyclic subgroup $C$ of $N$, then
$m_{G, N}= 0$.
\end{proposition}

\begin{proof}By Proposition 5.3, we have
\begin{eqnarray*}
M_{G,N}'
&=&-\sum_{\substack{C\leq N\\C~ is~ cyclic}}\varphi(|C|)-\sum_{Y\lneq G,Y\nleq N}M'_{Y, Y\cap N}\\
&=&-\sum_{\substack{C\leq N\\C~ is~ cyclic}}(1-\chi(|N(\mathfrak{T}_{C}(G, N))|)\cdot \varphi(|C|)).
\end{eqnarray*}
Since for each cyclic subgroup $C$ of $N$, we have $|N(\mathfrak{T}_{C}(G, N))|$ is contractible.
It implies $\chi(|N(\mathfrak{T}_{C}(G, N)))=1$, thus $M_{G,N}'=0$.

By the definition of $M_{G,N}'$, we know that
$$m_{G,N}+\frac{1}{|G|}M_{G,N}'=m_{G,G}=0.$$
So $m_{G,N}=0$.
\end{proof}


\section{\bf To compute $m_{G, N}$}

Let $G$ be a finite group and $N\unlhd G$. We will prove Main Theorem in this section.
Recall
$$m_{G, N}=\frac{1}{|G|}\sum_{\substack{XN= G\\ X\leq G}} |X|\mu(X, G).$$
Let $H\lneq G$, and we set
$$m_{G,H}':=\frac{1}{|G|}\sum_{X\leq H} |X|\mu(X, G);$$
and set
$$M_{G,H}':=\sum_{X\leq H} |X|\mu(X, G)=|G|m_{G,H}'.$$

 Let $\{H_1, H_2,\cdots, H_n\}$ be the set of all maximal subgroup
of $G$ such that $N\lneq H_i$. Let $J=\{1,2,\ldots, n\}$ and $\sigma$ be a non-empty subset of $J$.
Set $H_{\sigma}:=\bigcap_{j\in \sigma}H_j$.

\begin{theorem}Let $G$ be a finite group, $G$ not cyclic and $N\unlhd G$.
Then
$$ m_{G,N}=\frac{1}{|G|}\sum_{\substack{C\leq G\\  C~ \mathrm{is~ cyclic}}}\sum_{i=1}^n
\sum_{\substack{\sigma\leq J\\ |\sigma|=i\\  C\leq H_\sigma}}
(-1)^i\widetilde{\chi}(|N(\mathfrak{T}_{C}(G, H_\sigma))|)\cdot \varphi(|C|).$$
Here, $|N(\mathfrak{T}_{C}(G,H_\sigma))|$ is a simplicial complex associated to the poset $\mathfrak{T}_{C}(G,H_\sigma)$, and
$\widetilde{\chi}(|N(\mathfrak{T}_{C}(G,H_\sigma))|)$ is the reduced Euler characteristic of the space $|N(\mathfrak{T}_{C}(G, H_\sigma))|$.
\end{theorem}

\begin{proof}Since $G$ is not cyclic, we have
\begin{eqnarray*}
0
&=&\frac{1}{|G|}\sum_{X\leq G} |X|\mu(X, G)\\
&=&\frac{1}{|G|}\sum_{\substack{XN= G\\ X\leq G}} |X|\mu(X, G)+\frac{1}{|G|}\sum_{\substack{XN\neq G\\ X\leq G}} |X|\mu(X, G).
\end{eqnarray*}
If $X\leq G$ and $XN\neq G$, thus there exists a maximal subgroup $H$ of $G$ such that
$$X\leq XN \leq H\lneq G.$$
Let $\{H_1, H_2,\cdots, H_n\}$ be the set of all maximal subgroup
of $G$ such that $N\leq H_i$. It implies that
$$\sum_{\substack{XN\neq G\\ X\leq G}} |X|\mu(X, G)=
\sum_{\substack{X\leq H_i\\ for ~some~ i} }|X|\mu(X,G).$$
Now, we focus on $$\sum_{\substack{X\leq H_i\\ for ~some~ i}}|X|\mu(X,G).$$
 By the inclusion-exclusion principle, we can see
\begin{eqnarray*}
&~&\sum_{\substack{X\leq H_i\\ for ~some~ i} }\mu(X,G)\\
&=&\sum_{i=1}^n M'_{G, H_i}-\sum_{1\leq i\lneq j\leq n}M'_{G, H_i\cap H_j}\\
&~&+\sum_{1\leq i\lneq j\lneq k\leq n}M'_{G, H_i\cap H_j\cap H_k}+\cdots+(-1)^{n+1}\cdot M'_{G, \bigcap_{i=1}^{n}H_i}\\
&=&\sum_{i=1}^n\sum_{\substack{C\leq H_i\\ C~ \mathrm{is~ cyclic}}}\widetilde{\chi}(|N(\mathfrak{T}_{C}(G, H_i))|)\cdot \varphi(|C|)+\cdots+\\
&~&(-1)^{n+1}\sum_{\substack{C\leq \bigcap_{i=1}^{n}H_i\\ C~ \mathrm{is~ cyclic}}}\widetilde{\chi}(|N(\mathfrak{T}_{C}(G, \bigcap_{i=1}^{n}H_i))|)\cdot \varphi(|C|)\\
&=&-\sum_{\substack{C\leq G\\  C~ \mathrm{is~ cyclic}}}\sum_{i=1}^n
\sum_{\substack{\sigma\leq J\\ |\sigma|=i\\  C\leq H_\sigma}}
(-1)^i\widetilde{\chi}(|N(\mathfrak{T}_{C}(G, H_\sigma))|)\cdot \varphi(|C|).
\end{eqnarray*}
Hence, we have $$ m_{G,N}=\frac{1}{|G|}\sum_{\substack{C\leq G\\  C~ \mathrm{is~ cyclic}}}\sum_{i=1}^n
\sum_{\substack{\sigma\leq J\\ |\sigma|=i\\  C\leq H_\sigma}}
(-1)^i\widetilde{\chi}(|N(\mathfrak{T}_{C}(G, H_\sigma))|)\cdot \varphi(|C|).$$
\end{proof}

\begin{remark} Since
$$m_{G, N}=\frac{1}{|G|}\sum_{\substack{XN= G\\ X\leq G}} |X|\mu(X, G)=-\frac{1}{|G|}
\sum_{\substack{X\leq H\\ for ~some\\ N\leq  H\lneq G} }|X|\mu(X,G),$$
thus we can see that $m_{G,N}$ depends on $H$ with $N\leq  H\lneq G$.
So it may be a reason why there exists a relation between $G$ and $\beta(G)$.
\end{remark}

\textbf{ACKNOWLEDGMENTS}\hfil\break
The authors would like to thank Prof. S. Bouc for his numerous discussion in Beijing in Oct. 2014.
And the second author would like to thank  Prof. C. Broto for his constant encouragement in Barcelona in Spain.
The authors would like to thank the reviewer of \cite{LXZ}, and the proof of Lemma 3.1 and
 the formula $(\ast)$ of  Proposition 4.1 are due to the reviewer.

\end{document}